\documentclass[11pt,reqno]{amsart}
\usepackage{indentfirst, amssymb,amsmath,amsthm, mathrsfs,cite,graphicx,float}
\newtheorem{theo}{Theorem}[section]

\newtheorem{open problem}{Open problem}[section]
\newcommand{\pa}{\partial}

\newcommand{\be}{\begin{equation}}
\newcommand{\ee}{\end{equation}}
\newcommand{\bs}{\begin{small}}
\newcommand{\es}{\end{small}}
\newcommand{\beas}{\begin{eqnarray*}}
\newcommand{\eeas}{\end{eqnarray*}}
\newcommand{\bea}{\begin{eqnarray}}
\newcommand{\eea}{\end{eqnarray}}
\renewcommand{\epsilon}{\varepsilon}
\numberwithin{equation}{section}
\begin{document}

\title[Estimation of the pre-Schwarzian norm]{An estimation of the pre-Schwarzian norm for certain classes of analytic functions}
\author[V. Allu, R. Biswas and R. Mandal]{Vasudevarao Allu, Raju Biswas and Rajib Mandal}
\date{}
\address{Vasudevarao Allu, Department of Mathematics, School of Basic Science, Indian Institute of Technology Bhubaneswar, Bhubaneswar-752050, Odisha, India.}
\email{avrao@iitbbs.ac.in}
\address{Raju Biswas, Department of Mathematics, Raiganj University, Raiganj, West Bengal-733134, India.}
\email{rajubiswasjanu02@gmail.com}
\address{Rajib Mandal, Department of Mathematics, Raiganj University, Raiganj, West Bengal-733134, India.}
\email{rajibmathresearch@gmail.com}
\maketitle
\let\thefootnote\relax
\begin{abstract}
The primary objective of this paper is to establish the sharp estimates of the pre-Schwarzian norm for functions $f$ in the class 
$\mathcal{S}^*(\varphi)$ and $\mathcal{C}(\varphi)$ when $\varphi(z)=1/(1-z)^s$ with $0<s\leq 1$ and $\varphi(z)=(1+sz)^2$ with $0<s\leq 1/\sqrt{2}$, where $\mathcal{S}^*(\varphi)$ and $\mathcal{C}(\varphi)$ are the Ma-Minda type starlike and Ma-Minda type convex classes associated with $\varphi$, respectively.
\end{abstract}
\section{Introduction}
\noindent Let $\mathcal{H}$ denote the class of all analytic functions in the unit disk $\mathbb{D}:=\{z\in\mathbb{C}:|z|<1\}$ and let $\mathcal{A}$ denote the class of functions $f\in\mathcal{H}$ of the form
\bea\label{R-01} f(z)= z+\sum_{n=2}^{\infty}a_n z^n.\eea
Further, let $\mathcal{S}$ be the subclass  of $\mathcal{A}$ that are univalent ({\it i.e.}, one-to-one) in $\mathbb{D}$. 
A domain $\Omega$ is called starlike with respect to a point $z_0\in\Omega$ if the line segment joining $z_0$ to any point in $\Omega$ lies in $\Omega$, {\it i.e.}, 
$(1-t)z_0+t z\in\Omega$ for all $t\in[0,1]$ and for all $z\in\Omega$. In particular, if $z_0=0$, then $\Omega$ is simply called starlike. A function 
$f\in\mathcal{A}$ is said to be starlike if $f(\mathbb{D})$ is starlike with respect to the origin. 
Let $\mathcal{S}^*$ denote the class of starlike functions in $\mathbb{D}$. It is well-known that a function $f\in\mathcal{A}$ is in $\mathcal{S}^*$ if, and only if, $\text{Re} (zf'(z)/f(z))>0$ for $z\in\mathbb{D}$. A domain $\Omega$ is called convex if it is starlike with respect to any point 
in $\Omega$. In other words, convexity implies starlikeness, but the converse is not necessarily true. A domain can be starlike 
without being convex. A function $f\in\mathcal{A}$ is said to be convex if $f(\mathbb{D})$ is convex. Let  $\mathcal{C}$ denote the class of convex functions in $\mathbb{D}$.
It is well-known that a function $f\in\mathcal{A}$ is in $\mathcal{C}$ if, and only if, $\text{Re}\left(1+zf''(z)/f'(z)\right)>0$ for $z\in\mathbb{D}$. Moreover, a function 
$f\in\mathcal{A}$ is said to be $\alpha$-spirallike function if $\textrm{Re}(e^{-i\alpha}zf'(z)/f(z))>0$ for $z\in\mathbb{D}$, where $-\pi/2<\alpha<\pi/2$.
For more details about the aforementioned classes, we refer to \cite{D1983, G1983, TTA2018}.\\[2mm]
\indent Let $\mathcal{B}$ be the class of all analytic functions $\omega:\mathbb{D}\rightarrow\mathbb{D}$ and $\mathcal{B}_0=\{\omega\in\mathcal{B} : \omega(0)=0\}$. 
Functions in $\mathcal{B}_0$ are called Schwarz function. According to Schwarz's lemma, if $\omega\in\mathcal{B}_0$, then $|\omega(z)|\leq |z|$ and $|\omega'(0)|\leq 1$. 
Strict inequality holds in both estimates unless $\omega(z)=e^{i\theta}z$, $\theta\in\mathbb{R}$. A sharpened form of the Schwarz lemma, known as the Schwarz-Pick lemma, 
gives the estimate  $|\omega'(z)|\leq (1-|\omega(z)|^2)/(1-|z|^2)$ for $z\in\mathbb{D}$ and $\omega\in\mathcal{B}$.\\[2mm]
\indent An analytic function $f$ in $\mathbb{D}$ is said to be subordinate to an analytic function $g$ in $\mathbb{D}$, written as $f\prec g$, if there exists a function $\omega\in\mathcal{B}_0$ such that $f(z)=g(\omega(z))$ for $z\in\mathbb{D}$. Moreover, if $g$ is univalent in $\mathbb{D}$,
then $f \prec g$ if, and only if, $f(0)=g(0)$ and $f(\mathbb{D})\subseteq g(\mathbb{D})$. For basic details and results on subordination classes, we refer
to \cite[Chapter 6]{D1983}. 
Using the notion of subordination, Ma and Minda \cite{MM1992} have introduced more general subclasses of starlike and convex functions as follows:
\beas\mathcal{S}^*(\varphi)=\left\{f\in\mathcal{S}:\frac{zf'(z)}{f(z)}\prec\varphi(z)\right\}\quad\text{and}\quad\mathcal{C}(\varphi)=\left\{f\in\mathcal{S}:1+\frac{zf''(z)}{f'(z)}\prec\varphi(z)\right\},\eeas
where the function $\varphi :\mathbb{D}\to\mathbb{ C}$, called Ma-Minda function, is analytic and univalent in $\mathbb{D}$ such that $\varphi(\mathbb{D})$ has positive real 
part, symmetric with respect to the real axis, starlike with respect to $\varphi(0)=1$ and $\varphi'(0)>0$. A Ma-Minda function has the Taylor series expansion 
of the form $\varphi(z)=1+\sum_{n=1}^\infty a_nz^n $ $(a_1>0)$.
We call $\mathcal{S}^*(\varphi)$ and $\mathcal{C}(\varphi)$ the Ma-Minda type starlike and Ma-Minda type convex classes associated with $\varphi$, respectively. One 
can easily prove the inclusion 
relations $\mathcal{S}^*(\varphi)\subset\mathcal{S}^*$ and $\mathcal{C}(\varphi)\subset\mathcal{C}$. It is known that $f\in\mathcal{C}(\varphi)$ if, and only if, $zf'\in\mathcal{S}^*(\varphi)$.\\[2mm]
\indent For different choices of the function $\varphi$, the classes $\mathcal{S}^*(\varphi)$ and $\mathcal{C}(\varphi)$ generate several important subclasses of $\mathcal{S}^*$
 and $\mathcal{C}$, respectively. 
For example, if $\varphi(z)=(1+z)/(1-z)$, then $\mathcal{S}^*(\varphi)=\mathcal{S}^*$ and $\mathcal{C}(\varphi)=\mathcal{C}$. For $\varphi(z)=(1+(1-2\alpha)z)/(1-z),~0\leq \alpha<1$, we get the classes $\mathcal{S}^*(\alpha)$ of starlike functions of 
order $\alpha$ and $\mathcal{C}(\alpha)$ of convex functions of order $\alpha$. If $\varphi=((1+z)/(1-z))^\alpha$ for $0< \alpha\leq 1$, then 
$\mathcal{S}^*(\varphi)=\mathcal{S}\mathcal{S}^*(\alpha)$ the class of strongly starlike functions of order $\alpha$ and 
$\mathcal{C}(\varphi)=\mathcal{S}\mathcal{C}(\alpha)$ the class of strongly convex functions of order $\alpha$ (see \cite{S1966}). Also for  $\varphi=(1+Az)/(1+Bz),~-1\leq B<A\leq 1$, we have the classes 
of Janowski starlike functions $\mathcal{S}^*[A,B]$ and Janowski convex functions $\mathcal{C}[A,B]$ (see \cite{J1973}).
For $\varphi(z)=(1+2/\pi^2(\log(1-\sqrt{z})/(1+\sqrt{z}))^2)$ the class $\mathcal{C}(\varphi)$ (resp., $\mathcal{S}^*(\varphi)$) is the class {\it UCV} (resp. {\it UST} ) of 
normalized uniformly convex (resp. starlike) functions (see \cite{R1993,1G1991,2G1991, R1991}). Ma and Minda \cite{1MM1992, MM1993} have studied the class {\it UCV} 
extensively. Cho {\it et al.} \cite{CKKR2019} introduced the family $\mathcal{S}^*(1+\sin z)$ and studied the radius of starlikeness and convexity. Kargar {\it et al.} \cite{KES2019}
have introduced the class $\mathcal{BS}^*(\alpha):=\mathcal{S}^*(1+z/(1-\alpha z^2))$, which is associated with the Booth lemniscate.
The class $\mathcal{S}^*(2/(1+e^{-z})$ was introduced by Goel and Kumar \cite{GK2020} and studied several inclusion relations, radius problems as well as coefficient estimates.\\[2mm]
\indent In this paper, we consider two different classes of functions: $\mathcal{S}^*_{hyp}=\mathcal{S}^*(\varphi_s)$ with $\varphi_s(z)=1/(1-z)^s$ $(0<s\leq 1)$ and $\mathcal{S}_{L}^*=\mathcal{S}^*(\varphi)$ with $\varphi(z)=(1+s z)^2$ ($0<s\leq 1/\sqrt{2}$), where the branch of the logarithm is determined by $\varphi_s(0)=1$. More precisely,
\beas \label{R-05}
&&\mathcal{S}^*_{hyp}=\left\{f\in\mathcal{A} :\frac{zf'(z)}{f(z)}\prec \frac{1}{(1-z)^s},~ 0<s\leq 1\right\}\\\text{and}
&&\label{R-10}\mathcal{S}^*_{L}=\left\{f\in\mathcal{A}:\frac{zf'(z)}{f(z)}\prec (1+s z)^2, 0<s\leq 1/\sqrt{2}\right\}.
\eeas
The function 
\beas \frac{1}{(1-z)^s}=\mathrm{exp}(-s\log(1-z))=1+\sum_{n=1}^\infty \frac{s(s+1)\cdots(s+n-1)}{n!}z^n\quad (z\in\mathbb{D}),\eeas
 where the branch of the logarithm is determined by $\log(1)=0$.
It is evident that the function $\varphi(z)=1/(1-z)^s$ maps the unit disk $\mathbb{D}$ onto a domain bounded by the right branch of the hyperbola
\beas H(s):=\left\{r e^{i\theta}: r=\frac{1}{(2\cos(\theta/s))^s},~ |\theta|<\frac{\pi s}{2}\right\},\eeas
as illustrated in Figure \ref{Fig1}.
Moreover, $\varphi(\mathbb{D})$ is symmetric respecting the real axis, $\varphi$ is convex and hence starlike with
respect to $\varphi(0)=1$. It is evident that $\varphi'(0)>0$ and $\varphi$ has positive real part in $\mathbb{D}$. Thus, $\varphi$ satisfies the category of Ma-Minda functions.
A function $f\in \mathcal{S}^*_{hyp}$ if, and only if, there exists an analytic function $p$ with $p(0)=1$ and $p(z)\prec 1/(1-z)^s$ in $\mathbb{D}$ such that
\bea\label{e1} f(z)=z\;\mathrm{exp}\left(\int_0^z \frac{p(t)-1}{t}dt\right).\eea
If we choose, $p(t)=1/(1-t)^s$, then from (\ref{e1}), we obtain the function 
\beas f_{s,1}(z):=z\;\mathrm{exp}\left(\int_0^z \frac{(1-t)^{-s}-1}{t}dt\right)=z+s z^2+\frac{3s^2+s}{4}z^3+\frac{17s^3+15s^2+4s}{36}z^4+\cdots.\eeas
\begin{figure}[H]
\begin{minipage}[c]{0.5\linewidth}
\centering
\includegraphics[scale=0.7]{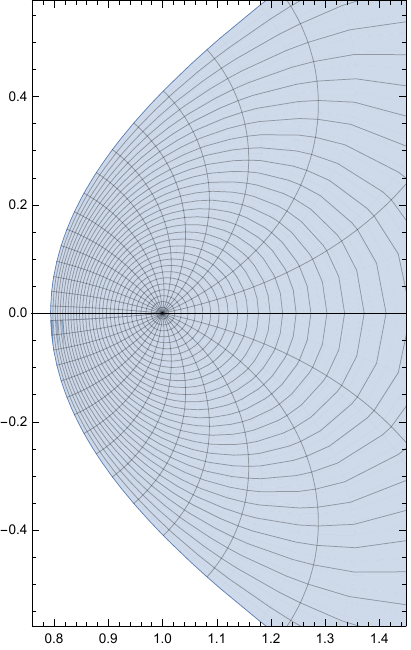}
\caption{Image of $\mathbb{D}$ under the mapping $1/(1-z)^s$ for $s=1/3$}
\label{Fig1}
\end{minipage}
\begin{minipage}[c]{0.49\linewidth}
\centering
\includegraphics[scale=0.7]{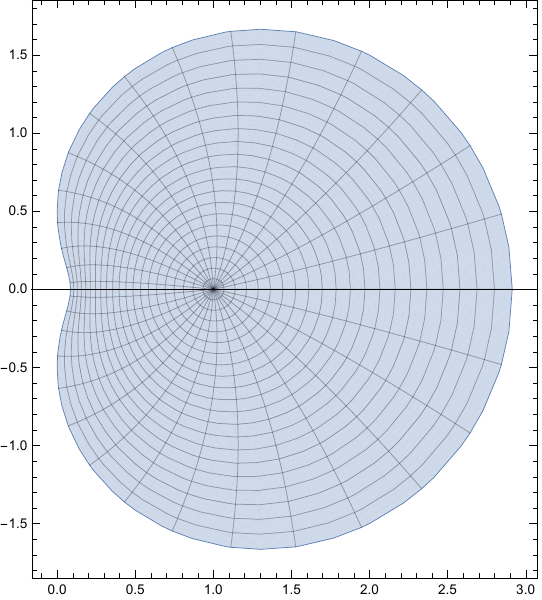}
\caption{Image of $\mathbb{D}$ under the mapping $(1+sz)^2$ for $s=1/\sqrt{2}$}
\label{Fig2}
\end{minipage}
\end{figure}
\noindent The function $\varphi(z)=(1+sz)^2$ maps the unit disk $\mathbb{D}$ onto a domain bounded by a lima\c{c}on given by
\beas \left\{u+iv\in\mathbb{C}: \left((u-1)^2+v^2-s^4\right)^2=4s^2\left(\left(u-1+s^2\right)^2+v^2\right)\right\},\eeas
which is symmetric about the real axis, as illustrated in Figure \ref{Fig2}. Note that for $0<s\leq 1/\sqrt{2}$, $\varphi(z)$ satisfies the category of Ma-Minda functions.
A function $f\in \mathcal{S}^*_{L}$ if, and only if, there exists 
an analytic function $p$ with $p(0)=1$ and $p\prec (1+s z)^2$ in $\mathbb{D}$ such that
\bea\label{e3} f(z)=z\;\mathrm{exp}\left(\int_0^z \frac{p(t)-1}{t}dt\right).\eea 
If we choose $p(z)=(1+sz)^2$ in \eqref{e3}, we obtain
\beas f_{s,2}(z)=z\;\mathrm{exp}\left(\int_0^z\frac{(1+st)^2-1}{t}dt\right)=z+2sz^2+\frac{5}{2}s^2z^3+\cdots.\eeas
The functions $f_{s,1}(z)$ and $f_{s,2}(z)$ plays the role of extremal function for many extremal problems in the classes $\mathcal{S}^*_{hyp}$ and $\mathcal{S}^*_{L}$, respectively. \\[2mm]
\indent
It is important to note that $f\in\mathcal{C}_{hyp}$ (resp. $\mathcal{C}_{L}$) if, and only if, $zf'\in\mathcal{S}^*_{hyp}$ (resp. $\mathcal{S}^*_{L}$), where  the classes $\mathcal{C}_{hyp}$ and $\mathcal{C}_{L}$ are defined by
\beas&&\mathcal{C}_{hyp}=\left\{f\in\mathcal{A}:1+\frac{zf''(z)}{f'(z)}\prec \frac{1}{(1-z)^s}, ~0<s\leq 1\right\}\\\text{and}
&&\mathcal{C}_{L}=\left\{f\in\mathcal{A}:1+\frac{zf''(z)}{f'(z)}\prec (1+sz)^2,~0<s\leq 1/\sqrt{2}\right\}.\eeas
It is evident that a function $f\in \mathcal{C}_{hyp}$ if, and only if, there exists an analytic function $p$ with $p(0)=1$ and $p(z)\prec 1/(1-z)^s$ in $\mathbb{D}$ such that
\bea\label{e5} f(z)=\int_0^z\left(\;\mathrm{exp}\left(\int_0^u \frac{p(t)-1}{t}dt\right)\right)du.\eea 
If we choose $p(z)=1/(1-z)^s$ in \eqref{e5}, we obtain
\beas f_{s,3}(z)=\int_0^z\mathrm{exp}\left(\int_0^u \frac{(1-t)^{-s}-1}{t}dt\right)du\in \mathcal{C}_{hyp}.\eeas
Note that $zf_{s,3}'(z)=f_{s,1}(z)$.
 A function $f\in \mathcal{C}_{L}$ if, and only if, there exists 
an analytic function $p$ with $p(0)=1$ and $p\prec(1+sz)^2$ in $\mathbb{D}$ such that
\bea\label{e6} f(z)=\int_0^z\left(\;\mathrm{exp}\left(\int_0^u \frac{p(t)-1}{t}dt\right)\right)du.\eea 
If we choose $p(z)=(1+sz)^2$ in \eqref{e6}, we obtain
\beas f_{s,4}(z)=\int_0^z\mathrm{exp}\left(\int_0^u \frac{(1+s t)^2-1}{t}dt\right)du\in \mathcal{C}_{L}.\eeas
It is evident that $zf_{s,4}'(z)=f_{s,2}(z)$. For a more in-depth results of these classes, we refer to \cite{KME2019,KME2020,EBCA2020,MK2020,BR2023,CSW2022}.
 \section{Pre-Schwarzian Norm}
An analytic function $f(z)$ in a domain $\Omega$ is said to be locally univalent if for each $z_0\in\Omega$, there exists a neighborhood $U$ of $z_0$ such that $f(z)$ is univalent in $U$. It is well-known that the non-vanishing of the Jacobian is necessary and sufficient conditions for local univalence (see \cite[Chapter 1]{D1983}).
Let $\mathcal{LU}$ denote the subclass of $\mathcal{H}$ consisting of all locally univalent functions in $\mathbb{D}$, {\it i.e.}, $\mathcal{LU}:=\{f\in\mathcal{H}:f'(z)\not= 0~~\text{ for all }~~z\in\mathbb{D}\}$. For $f\in\mathcal{LU}$, the pre-Schwarzian derivative is defined by
\beas P_f(z):=\frac{f''(z)}{f'(z)},\eeas
and the pre-Schwarzian norm (the hyperbolic sup-norm) is defined by
\beas \Vert P_f\Vert :=\sup_{z\in\mathbb{D}}\;(1-|z|^2)\left|P_f(z)\right|.\eeas
This norm plays an important rule in the theory of Teichm\"{u}ller spaces. For a univalent function $f$ in $\mathbb{D}$, it is well-known that $\Vert P_f \Vert\leq 6$ and the equality 
is attained for the Koebe function or its rotation. One of the most used univalence criterion for locally univalent analytic functions is the Becker's univalence criterion \cite{B1972}, 
which states that if $f\in\mathcal{LU}$ and $\sup_{z\in\mathbb{D}}\left(1-|z|^2\right) \left|zP_f(z)\right|\leq1$, then $f$ is univalent in $\mathbb{D}$. In a subsequent study, 
Becker and Pommerenke \cite{BP1984} prove that the constant $1$ is sharp. In 1976, Yamashita \cite{Y1976} 
proved that $\Vert P_f \Vert<\infty $ is finite if, and only if, $f$ is uniformly 
locally univalent in $\mathbb{D}$. Moreover, if $\Vert P_f\Vert<2$, then $f$ is bounded in $\mathbb{D}$ (see \cite{KS2002}).\\[2mm]
\indent In the field of univalent function theory, several researchers have studied the pre-Schwarzian norm for various subclasses of analytic and univalent functions.
In 1998, Sugawa \cite{S1998} established the sharp estimate of the pre-Schwarzian norm for functions in the class of strongly starlike functions of order $\alpha$ ($0<\alpha\leq 1$).
In $1999$, Yamashita \cite{Y1999} proved that $\Vert P_f\Vert \leq 6-4\alpha$ for $f\in\mathcal{S}^*(\alpha)$ and $\Vert P_f\Vert \leq 4(1-\alpha)$ for $f\in\mathcal{C}(\alpha)$, 
where $0\leq \alpha<1$ and both the estimates are sharp. In $2000$, Okuyama \cite{O2000} established the sharp estimate of the pre-Schwarzian norm for $\alpha$-spirallike functions. Kim and Sugawa \cite{KS2006} established the sharp 
estimate of the pre-Schwarzian norm $\Vert P_f\Vert \leq 2(A-B)/(1+\sqrt{1-B^2})$ for $f\in\mathcal{C}[A,B]$ (see also \cite{PS2008}). Ponnusamy and Sahoo \cite{PS2010} 
obtained the sharp estimates of the pre-Schwarzian norm for functions in the class
$\mathcal{S}^*[\alpha,\beta]:=S^*\left(\left((1+(1-2\beta)z)/(1-z)\right)^\alpha\right)$, where $0<\alpha\leq 1$ and $0\leq \beta<1$. In $2014$, Aghalary and Orouji \cite{AO2014} obtained the sharp estimate of the pre-Schwarzian norm for $\alpha$-spirallike function of order $\rho$, where $\alpha\in(-\pi/2,\pi/2)$ and $\rho\in[0,1)$. The pre-Schwarzian norm of certain integral transform of $f$ for certain subclass of $f$ has been also studied in the literature. For a detailed study on pre-Schwarzian norm, we refer to \cite{KPS2004,PPS2008,PS2008, AP2023,1AP2023,2AP2023,AP2024,CKPS2005} and the references therein.\\[2mm]
\indent In this paper, we establish sharp estimates of the pre-Schwarzian norms for functions in the classes  $\mathcal{S}^*_{hyp}$, $\mathcal{S}^*_{L}$, $\mathcal{C}_{hyp}$ and $\mathcal{C}_{L}$.
\section{Main results}
In the following result, we establish the sharp estimate of the pre-Schwarzian norm for functions $f$ in the class $\mathcal{S}^*_{hyp}$.
\begin{theo}\label{Th1} Let $f\in\mathcal{S}^*_{hyp}$. Then the pre-Schwarzian norm satisfies the following sharp inequality
\beas\Vert P_f\Vert \leq \left\{\begin{array}{lll}
\dfrac{s t_s(1+t_s)+(1+t_s)(1-t_s)^{1-s}-(1-t_s^2)}{t_s}&\text{for}~s\in(0,1)\\[2mm]
4&\text{for}~s=1,\end{array}\right.\eeas
where $t_s\in(0,1)$ is the unique root of the equation
\beas \frac{(1-t)^{-s} \left(s t^2 (1-t)^s+t^2 (1-t)^s+s t^2+(1-t)^s+s t-t^2-1\right)}{t^2}=0.\eeas
\end{theo}
\begin{proof}
Let $f\in\mathcal{S}^*_{hyp}$. By the definition of the class $\mathcal{S}^*_{hyp}$, we have
\beas\frac{zf'(z)}{f(z)}\prec \frac{1}{(1-z)^s}.\eeas
Thus, there exist a Schwarz function $\omega\in\mathcal{B}_0$ such that
\beas\frac{zf'(z)}{f(z)}=\frac{1}{(1-\omega(z))^s}.\eeas
Taking logarithmic derivative on both sides with respect to $z$, we obtain
\beas
P_f(z)=\frac{f''(z)}{f'(z)}=\frac{s\omega'(z)}{1-\omega(z)}+\frac{1}{z}\left(\frac{1}{(1-\omega(z))^s}-1\right).\eeas
Since $|\omega(z)|\leq |z|<1$ and the branch of the logarithm is determined by $\log(1)=0$, thus, we have
\beas \frac{1}{(1-\omega(z))^s}=\mathrm{exp}(-s\log(1-\omega(z)))=1+\sum_{n=1}^\infty \frac{s(s+1)\cdots(s+n-1)}{n!}\omega^n(z)\quad (z\in\mathbb{D}).\eeas
In view of the Schwarz-Pick lemma, we have
\beas (1-|z|^2)|P_f(z)|&=&(1-|z|^2)\left|\frac{s\omega'(z)}{1-\omega(z)}+\frac{1}{z}\left(\frac{1}{(1-\omega(z))^s}-1\right)\right|\\
&&\leq (1-|z|^2)\left(\frac{s|\omega'(z)|}{1-|\omega(z)|}+\frac{1}{|z|}\left(\frac{1}{(1-|\omega(z)|)^s}-1\right)\right)\\
&&\leq\frac{s\left(1-|\omega(z)|^2\right)}{1-|\omega(z)|}+\frac{(1-|z|^2)}{|z|}\left(\frac{1}{(1-|\omega(z)|)^s}-1\right).\eeas
For $0\leq t:=|\omega(z)|\leq |z|<1$, we have
\beas (1-|z|^2)|P_f(z)|\leq s(1+t)+\frac{(1-|z|^2)}{|z|}\left(\frac{1}{(1-t)^s}-1\right).\eeas
Therefore, we have
\bea \label{e7}\Vert P_f\Vert=\sup_{z\in\mathbb{D}}\; (1-|z|^2)|P_f(z)|\leq \sup_{0\leq t\leq |z|<1} F_1(|z|,t),\eea
where 
\beas F_1(r, t)=s(1+t)+\frac{(1-r^2)}{r}\left(\frac{1}{(1-t)^s}-1\right)\quad\text{for}\quad r=|z|.\eeas
Now the objective is to determine the supremum of $F_1(r, t)$ on $\Omega=\{(r, t): 0< t\leq r<1\}$.
Differentiating partially $F_1(r, t)$ with respect to $r$, we obtain
\beas\frac{\pa}{\pa r}F_1(r, t)=-\left(\frac{1}{r^2}+1\right)\left(\frac{1}{(1-t)^s}-1\right)<0.\eeas
Therefore, $F_1(r, t)$ is a monotonically decreasing function of $r\in[t,1)$ and it follows that $F_1(r, t) \leq F_1(t, t)=F_2(t)$,
where
\beas F_2(t)=s(1+t)+\frac{(1-t^2)}{t}\left(\frac{1}{(1-t)^s}-1\right).\eeas
It is evident that $F_2(t)=2(1+t)$ for $s=1$. Hence, we have $\Vert P_f\Vert\leq 4$. We consider the case, where $0<s<1$.
Differentiating $F_2(t)$ with respect to $t$, we obtain
\beas &&F_2'(t)=\frac{(1-t)^{-s} \left(s t^2 (1-t)^s+t^2 (1-t)^s+s t^2+(1-t)^s+s t-t^2-1\right)}{t^2}\\\text{and}
&&F_2''(t)=\frac{s^2 t^3+s^2 t^2-s t^3+s t^2+2 t (1-t)^s-2 (1-t)^s-2 s t-2 t+2}{(1-t)^{s+1} t^3}.\eeas
Let 
\beas
F_3(t)=\frac{s^2 t^3+s^2 t^2-s t^3+s t^2+2 t (1-t)^s-2 (1-t)^s-2 s t-2 t+2}{(1-t)^{s+1}}.\eeas
It is evident that $F_2''(t)=F_3(t)/t^3$ and 
\beas 
F_3'(t)=\frac{-s(1-s) t^2(s t+s-2 t+4)}{(1-t)^{s+2}}<0\quad\text{for}\quad  0<t<1,~0<s<1.\eeas
Therefore, $F_3(t)$ is a monotonically decreasing function of $t\in(0,1)$ and it follows that $F_3(t)\leq \lim_{t\to 0^+}F_3(t)=0$, {\it i.e.,} $F_2''(t)\leq 0$ for $0<t<1$.
Thus, $F_2'(t)$ is a monotonically decreasing function in $t$ with $\lim_{t\to 0^+}F_2'(t)=s (s+3)/2$ and $\lim_{t\to 1^-}F_2'(t)=-\infty$. 
Therefore, the equation $F_2'(t)=0$ has the unique root $t_s$ in $(0,1)$, as illustrated in Figure \ref{Fig8}. Thus, $F_2(t)$ attains its maximum value at $t=t_s$. From (\ref{e7}), we have
\beas\Vert P_f\Vert \leq F_2(t_s)=\frac{s t_s(1+t_s)+(1+t_s)(1-t_s)^{1-s}-(1-t_s^2)}{t_s},\eeas
where $t_s\in(0,1)$ is the unique positive root of the equation 
\bea\label{e8} F_s(t):=\frac{(1-t)^{-s} \left(s t^2 (1-t)^s+t^2 (1-t)^s+s t^2+(1-t)^s+s t-t^2-1\right)}{t^2}=0.\eea
\indent To show that the estimate is sharp, we consider the function $f_1$ given by
\beas f_1(z)=z\;\mathrm{exp}\left(\int_0^z \frac{(1-t)^{-s}-1}{t}dt\right).\eeas
The pre-Schwarzian norm of $f_1$ is given by
\beas \Vert P_{f_1}\Vert =\sup_{z\in\mathbb{D}}\;(1-|z|^2)|P_{f_1}(z)|=\sup_{z\in\mathbb{D}}\;(1-|z|^2)\left|\frac{zs(1-z)^{-1}+(1-z)^{-s}-1}{z}\right|.\eeas
On the positive real axis, we note that
\beas&&\sup_{0\leq r<1}(1-r^2)\frac{rs(1-r)^{-1}+(1-r)^{-s}-1}{r}\\
&&=\left\{\begin{array}{lll}
\dfrac{s r_s(1+r_s)+(1+r_s)(1-r_s)^{1-s}-(1-r_s^2)}{r_s}&\text{for}~s\in(0,1)\\[2mm]
4&\text{for}~s=1,\end{array}\right.\eeas
where $r_s\in(0,1)$ is the unique root of the equation \eqref{e8}. Therefore,
\beas\Vert P_{f_1}\Vert =\left\{\begin{array}{lll}
\dfrac{s r_s(1+r_s)+(1+r_s)(1-r_s)^{1-s}-(1-r_s^2)}{r_s}&\text{for}~s\in(0,1)\\[2mm]
4&\text{for}~s=1.\end{array}\right.\eeas
This completes the proof.
\end{proof}
\noindent In Table \ref{tab0} and Figure \ref{Fig8}, we obtain the values of $t_{s}$ and $\Vert P_f\Vert$ for certain values of $s\in(0,1)$. We observe that, whenever $s\to1^-$, then $t_s\to1^{-}$ and $\Vert P_f\Vert\to4^-$.
\begin{table}[H]
\centering
\begin{tabular}{*{8}{|c}|}
\hline
$s$&1/2&1/3&2/3&3/4&4/5&9/10&99/100\\
\hline
$t_s$&0.765186&0.721166&0.819069&0.851565 &0.873603&0.926039&0.990582\\
\hline
$\Vert P_f\Vert$&1.45876&0.926878&2.06701&2.41553&2.64591&3.18262&3.86967\\
\hline
\end{tabular}
\caption{$t_s$ is the unique positive root of the equation (\ref{e8}) in $(0,1)$}
\label{tab0}\end{table}
\begin{figure}[H]
\centering
\includegraphics[scale=0.9]{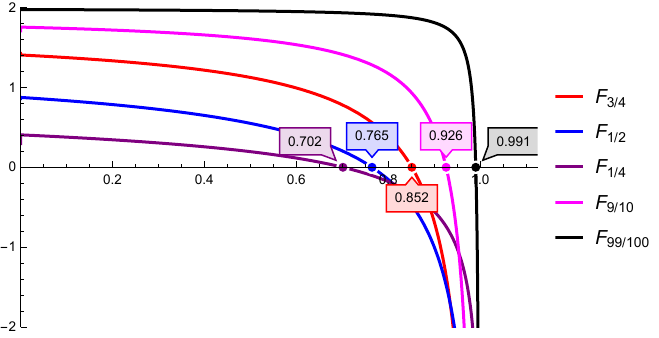}
\caption{Graph of $F_s(t)$ for different values of $s$ in $(0,1)$}
\label{Fig8}
\end{figure}
In the following result, we obtain the estimate of the pre-Schwarzian norm for functions in the class $\mathcal{S}^*_{L}$.
\begin{theo}
Let $f\in\mathcal{S}^*_{L}$. Then the pre-Schwarzian norm satisfies the following inequality
\beas\Vert P_f\Vert \leq \frac{2s (1-t_s^2)}{1-st_s}+\frac{(1-t_s^2)\left((1+st_s)^2-1\right)}{t_s},\eeas
where $t_s\in(0,1)$ is the unique positive root of the equation 
\beas\frac{-3 s^4 t^4+2 s^3 t^3+\left(s^4+7 s^2\right) t^2-\left(2 s^3+8 s\right) t+3 s^2}{(1-s t)^2}=0.\eeas
\end{theo}
\begin{proof}
Let $f\in\mathcal{S}^*_{L}$. By the definition of the class $\mathcal{S}^*_{L}$, we have
\beas\frac{zf'(z)}{f(z)}\prec (1+sz)^2.\eeas
Thus, there exists a Schwarz function $\omega(z)\in\mathcal{B}_0$ such that
\beas\frac{zf'(z)}{f(z)}=(1+s\omega(z))^2.\eeas
Taking logarithmic derivative on both sides with respect to $z$, we obtain
\beas P_f(z)=\frac{f''(z)}{f'(z)}=\frac{2s\omega'(z)}{1+s \omega(z)}+\frac{1}{z}\left((1+s\omega(z))^2-1\right),\eeas
In view of the Schwarz-Pick lemma, we have
\beas
(1-|z|^2)|P_f(z)|&\leq& (1-|z|^2)\left(\frac{2s|\omega'(z)|}{|1+s\omega(z)|}+\frac{|(1+s\omega(z))^2-1|}{|z|}\right)\\
&\leq&(1-|z|^2)\left(\frac{2s|\omega'(z)|}{1-s|\omega(z)|}+\frac{(1+s|\omega(z)|)^2-1}{|z|}\right)\\
&\leq&\frac{2s\left(1-|\omega(z)|^2\right)}{1-s|\omega(z)|}+\frac{(1-|z|^2)\left((1+s|\omega(z)|)^2-1\right)}{|z|}.\eeas
For $0\leq t:=|\omega(z)|\leq |z|<1$, we obtain
\beas(1-|z|^2)|P_f(z)|\leq \frac{2s (1-t^2)}{1-st}+\frac{(1-|z|^2)\left((1+st)^2-1\right)}{|z|}.\eeas
Therefore, we have
\bea\label{e9}\Vert P_f\Vert=\sup_{z\in\mathbb{D}}\;(1-|z|^2)|P_f(z)|\leq \sup_{0\leq t\leq|z|<1}F_4(|z|, t),\eea
where 
\beas F_4(r, t)=\frac{2s (1-t^2)}{1-st}+\frac{(1-r^2)\left((1+st)^2-1\right)}{r}\quad \text{for}\quad |z|=r.\eeas
Now our objective is to determine the supremum of $F_4(r, t)$ on $\Omega=\{(r, t): 0< t\leq r<1\}$.
Differentiating partially $F_4(r, t)$ with respect to $r$, we obtain
\beas\frac{\pa}{\pa r}F_4(r, t)=-\left(\frac{1}{r^2}+1\right)\left((1+st)^2-1\right)<0.\eeas
Therefore, $F_4(r, t)$ is a monotonically decreasing function of $r\in[t,1)$ and it follows that $F_4(r, t) \leq F_4(t, t)=F_5(t)$,
where
\bea\label{e10} F_5(t)=\frac{2s (1-t^2)}{1-st}+\frac{(1-t^2)\left((1+st)^2-1\right)}{t}.\eea
Differentiate $F_5(t)$ twice with respect to $t$, we obtain
\beas F_5'(t)&=&\frac{-3 s^4 t^4+2 s^3 t^3+\left(s^4+7 s^2\right) t^2-\left(2 s^3+8 s\right) t+3 s^2}{(1-s t)^2},\\
F_5''(t)&=&\frac{2\left(3 s^5 t^4-7 s^4 t^3+3 s^3 t^2+3 s^2 t+2 s^3-4 s\right)}{(1-st)^3}.\eeas
Let 
\beas F_6(t)=3 s^5 t^4-7 s^4 t^3+3 s^3 t^2+3 s^2 t+2 s^3-4 s.\eeas
Differentiating $F_6(t)$ with respect to $t$, we obtain
\beas F_6'(t)=3s^2 \left(4 s^3 t^3-7 s^2 t^2+2 s t+1\right)=3 s^2 (1-s t)^2 (1+4 s t)>0.\eeas
Therefore, $ F_6(t)$ is a monotonically increasing function of $t\in[0,1)$ and it follows that
\beas F_6(t)\leq F_6(1)=3 s^5-7 s^4 +3 s^3+3 s^2 +2 s^3-4 s=-s(1-s)\left(3 s^3-4 s^2+s+4\right)<0.\eeas
Therefore, $F_5''(t)<0$ and hence, we have $F_5'(t)$ is a monotonically decreasing function of $t$ with $F_5'(0)=3s^2$ and $\lim_{t\to1^-}F_5'(t)=(-2 s^4+10 s^2-8 s)/(1-s)^2=2 s\left(s^2+s-4\right)/ (1-s)<0$.
This leads us to conclude that the equation $F_5'(t)=0$ has the unique root $t_s$ in $(0,1)$. This shows that $F_5(t)$ attains its maximum at $t_s$. From \eqref{e9} and \eqref{e10}, we have
\beas \Vert P_f\Vert \leq F_5(t_s)=\frac{2s (1-t_s^2)}{1-st_s}+\frac{(1-t_s^2)\left((1+st_s)^2-1\right)}{t_s},\eeas
where $t_s\in(0,1)$ is the unique positive root of the equation 
\bea\label{f4}G_s(t):=\frac{-3 s^4 t^4+2 s^3 t^3+\left(s^4+7 s^2\right) t^2-\left(2 s^3+8 s\right) t+3 s^2}{(1-s t)^2}=0.\eea
This completes the proof.
\end{proof}
\noindent In Table \ref{tab2} and Figure \ref{Fig9}, we obtain the values of $t_{s}$ and $\Vert P_f\Vert$ for certain values of $s\in(0,1/\sqrt{2}]$.
\begin{table}[H]
\centering
\begin{tabular}{*{6}{|c}|}
\hline
$s$				&1/2    				&11/20    		&2/3			&     3/5		&$1/\sqrt{2}$\\
\hline
$t_s$				&0.19266			&0.213611		&0.265836	&0.235311		&0.285555\\
\hline
$\Vert P_f\Vert$&2.07478			&2.30104	&2.85492	&2.53348	&3.05755\\
\hline
\end{tabular}
\caption{$t_s$ is the unique positive root of the equation (\ref{f4}) in $(0,1)$}
\label{tab2}\end{table}
\begin{figure}[H]
\centering
\includegraphics[scale=0.9]{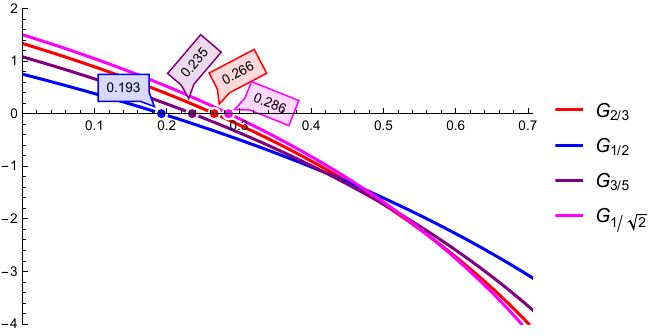}
\caption{Graph of $G_s(t)$ for different values of $s$ in $(0,1/\sqrt{2}]$}
\label{Fig9}
\end{figure}
In the following result, we establish the sharp estimate of the pre-Schwarzian norm for the functions in the class $\mathcal{C}_{hyp}$.
\begin{theo}
For any $g\in\mathcal{C}_{hyp}$, the pre-Schwarzian norm satisfies the following sharp inequality
\beas\Vert P_g\Vert  \leq  \left\{\begin{array}{lll}
\dfrac{(1+r_s)(1-r_s)^{1-s}-(1-r_s^2)}{r_s}&\text{for}~0<s<1\\[2mm]
2&\text{for}~s=1,\end{array}\right.\eeas
where $r_s\in(0,1)$ is the unique root of the equation 
\beas \frac{(1-r)^{-s} \left(r^2 (1-r)^s+r^2 s-r^2+(1-r)^s+r s-1\right)}{r^2}=0.\eeas
\end{theo}
\begin{proof}
Let $g\in\mathcal{C}_{hyp}$, then by the definition of the class $\mathcal{C}_{hyp}$, we have  
\beas 1+\frac{zg''(z)}{g'(z)}\prec \frac{1}{(1-z)^s}.\eeas
Thus, there exist a function $\omega(z)\in\mathcal{B}_0$ such that
\beas
1+\frac{zg''(z)}{g'(z)}=\frac{1}{(1-\omega(z))^s}.\eeas
By a simple calculation, we have
\beas(1-|z|^2)|P_g(z)|=(1-|z|^2)\left|\frac{g''(z)}{g'(z)}\right|&=&(1-|z|^2)\left|\frac{1}{z}\left(\frac{1}{(1-\omega(z))^s}-1\right)\right|.\eeas
Since $|\omega(z)|\leq |z|<1$ and the branch of the logarithm is determined by $\log(1)=0$, thus, we have
\beas \frac{1}{(1-\omega(z))^s}=\mathrm{exp}(-s\log(1-\omega(z)))=1+\sum_{n=1}^\infty \frac{s(s+1)\cdots(s+n-1)}{n!}\omega^n(z)\quad (z\in\mathbb{D}).\eeas
 Thus, we have
 \beas(1-|z|^2)|P_g(z)|\leq (1-|z|^2)\frac{1}{|z|}\left(\frac{1}{(1-|\omega(z)|)^s}-1\right)\leq (1-|z|^2)\frac{1}{|z|}\left(\frac{1}{(1-|z|)^s}-1\right).\eeas
Therefore, we have 
\bea\label{f1} \Vert P_g\Vert =\sup_{z\in\mathbb{D}}(1-|z|^2)|P_g(z)|\leq \sup_{0\leq |z|<1} F_6(|z|),\eea
where 
\beas F_6(r)=\frac{(1-r^2)}{r}\left(\frac{1}{(1-r)^s}-1\right)=\frac{(1+r)(1-r)^{1-s}-(1-r^2)}{r}\quad \text{for} \quad |z|=r.\eeas
It is easy to see that for $s=1$, we have $F_6(r)=1+r$. Hence, we have $\Vert P_g\Vert \leq 2$. Now, we consider the case, where $0<s<1$.
A simple calculation gives
\beas
&&F_6'(r)=\frac{(1-r)^{-s} \left(r^2 (1-r)^s+r^2 s-r^2+(1-r)^s+r s-1\right)}{r^2}\\\text{and}
&& F_6''(r)=\frac{r^3 s^2-r^3 s+r^2 s^2+r^2 s+2 r (1-r)^s-2 (1-r)^s-2 r s-2 r+2}{(1-r)^{s+1} r^3}.\eeas
Let 
\beas F_7(r)=(1-r)^{-s-1} \left(r^3 s^2-r^3 s+r^2 s^2+r^2 s+2 r (1-r)^s-2 (1-r)^s-2 r s-2 r+2\right).\eeas
It is evident that $ F_6''(r)=F_7(r)/r^3$ and 
\beas F_7'(r)=\frac{r^2 (s-1) s(r s-2 r+s+4)}{(1-r)^{s+2}}\leq 0\quad\text{for}\quad 0<r<1,~0<s<1.\eeas
Therefore, $F_7(r)$ is a monotonically decreasing function of $r\in(0,1)$ and it follows that $F_7(r)\leq \lim_{r\to0^+}F_7(r)=0$, {\it i.e.,} $F_6''(r)\leq 0$ for $0<r<1$.
Thus, $F_6'(r)$ is a monotonically decreasing function in $r$ with $\lim_{r\to 0^+}F_6'(r)=s(s+1)/2$ and $\lim_{r\to 1^-}F_6'(r)=-\infty$. 
Therefore, the equation $F_6'(r)=0$ has the unique root $r_s$ in $(0,1)$, as illustrated in Figure \ref{Fig3}. Thus, $F_6(r)$ attains its maximum value at $r=r_s$. 
From \eqref{f1}, we have
\beas \Vert P_g\Vert \leq \dfrac{(1+r_s)(1-r_s)^{1-s}-(1-r_s^2)}{r_s},\eeas
where $r_s\in(0,1)$ is the unique root of the equation 
\bea\label{f3} h_s(r):=\dfrac{(1-r)^{-s} \left(r^2 (1-r)^s+r^2 s-r^2+(1-r)^s+r s-1\right)}{r^2}=0.\eea
\indent To show that the estimate is sharp, we consider the function $f_3$ given by
\beas f_2(z)=\int_0^z \mathrm{exp}\left(\int_0^u \frac{(1-t)^{-s}-1}{t}dt\right)du.\eeas
The pre-Schwarzian norm of $f_2$ is given by
\beas \Vert P_{f_2}\Vert =\sup_{z\in\mathbb{D}}(1-|z|^2)|P_{f_2}(z)|=\sup_{z\in\mathbb{D}}(1-|z|^2)\left|\frac{(1-z)^{-s}-1}{z}\right|.\eeas
On the positive real axis, we note that
\beas\sup_{0\leq r<1}(1-r^2)\frac{(1-r)^{-s}-1}{r}=\left\{\begin{array}{lll}
\dfrac{(1+r_s)(1-r_s)^{1-s}-(1-r_s^2)}{r_s},&\text{when}~0<s<1\\[2mm]
2,&\text{when}~s=1,\end{array}\right.\eeas
where $r_s\in(0,1)$ is the unique root of the equation \eqref{f3}. Therefore,
\beas\Vert P_{f_2}\Vert =\left\{\begin{array}{lll}
\dfrac{(1+r_s)(1-r_s)^{1-s}-(1-r_s^2)}{r_s},&\text{when}~0<s<1\\[2mm]
2,&\text{when}~s=1.\end{array}\right.\eeas
This completes the proof.
\end{proof}
\noindent In Table \ref{tab1} and Figure \ref{Fig3}, we obtain the values of $r_{s}$ and $\Vert P_g\Vert$ for certain values of $s\in(0,1)$. We observe that, whenever $s\to1^-$, then $r_s\to1^{-}$ and $\Vert P_g\Vert\to2^-$.
\begin{table}[H]
\centering
\begin{tabular}{*{8}{|c}|}
\hline
$s$&1/2&1/3&1/4&2/3&3/4&4/5&9/10\\
\hline
$r_s$&0.54079&0.451833&0.412132&0.647789&0.71149&0.754417&0.855998\\
\hline
$\Vert P_g\Vert$&0.622369&0.390816&0.286101&0.900474&1.06896&1.18504&1.47402\\
\hline
\end{tabular}
\caption{$r_s$ is the unique positive root of the equation (\ref{f3}) in $(0,1)$}
\label{tab1}\end{table}
\begin{figure}[H]
\centering
\includegraphics[scale=0.8]{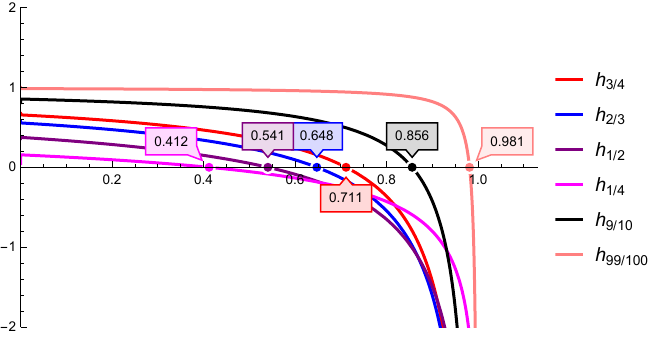}
\caption{Graph of $h_s(r)$ for different values of $s$ in $(0,1)$}
\label{Fig3}
\end{figure}
In the following result, we establish the sharp estimate of the pre-Schwarzian norm for the functions in the class $\mathcal{C}_{L}$.
\begin{theo}
For any $g\in\mathcal{C}_{L}$, the pre-Schwarzian norm satisfies the following sharp inequality
\beas\Vert P_g\Vert\leq \frac{2 \left(\sqrt{3 s^2+4}+4\right) \left(3 s^2+2 \sqrt{3 s^2+4}-4\right)}{27 s}.\eeas
\end{theo}
\begin{proof}
Let $g\in\mathcal{C}_{L}$, then by the definition of the class $\mathcal{C}_{L}$, we have
\beas 1+\frac{zg''(z)}{g'(z)}\prec (1+s z)^2.\eeas
Thus, there exist a function $\omega(z)\in\mathcal{B}_0$ such that
\beas1+\frac{zg''(z)}{g'(z)}=(1+s~\omega(z))^2,\quad\text{\it i.e.,}\quad \frac{g''(z)}{g'(z)}=\frac{(1+s~\omega(z))^2-1}{z}.\eeas
As $|\omega(z)|\leq |z|<1$, we have
\beas
(1-|z|^2)|P_g|=(1-|z|^2)\left|\frac{g''(z)}{g'(z)}\right|&=&(1-|z|^2)\left|\frac{(1+s~\omega(z))^2-1}{z}\right|\\
&\leq&(1-|z|^2)\left(\frac{(1+s~|\omega(z)|)^2-1}{|z|}\right)\\
&\leq&\frac{(1-|z|^2)\left((1+s|z|)^2-1\right)}{|z|}.\eeas
Therefore, the pre-Schwarzian norm for the function $g\in\mathcal{C}_{L}$ is
\bea\label{f2} \Vert P_g\Vert=\sup_{z\in\mathbb{D}}(1-|z|^2)|P_g(z)|\leq \sup_{0\leq|z|<1} F_8(|z|),\eea
where 
\beas F_8(r)=\frac{(1-r^2)\left((1+sr)^2-1\right)}{r}\quad \mathrm{for}~|z|=r.\eeas 
Differentiate twice $F_8(r)$ with respect to $r$, we obtain
\beas
F_8'(r)=s^2(1-3 r^2)-4 r s\quad \text{and} \quad
\quad F_8''(r)=-2 \left(3 r s^2+2 s\right)<0.\eeas
Thus, $F_8'(r)$ is a monotonically decreasing function in $r$ with $F_8'(0)=s^2$ and $F_8'(1)=-2 s^2-4 s$. 
Thus, the equation $F_8'(r)=0$ has the unique root $r_0=(-2+\sqrt{3 s^2+4})/(3 s)$ in $(0,1)$. Therefore, $F_8(r)$ attains its maximum value at $r=r_0$. 
\noindent From \eqref{f2}, we have
\beas\Vert P_g\Vert \leq \frac{2 \left(\sqrt{3 s^2+4}+4\right) \left(3 s^2+2 \sqrt{3 s^2+4}-4\right)}{27 s}.\eeas
\indent To show that the estimate is sharp, let us consider the function $f_4$ defined by
\beas f_3(z)=\int_0^z\mathrm{exp}\left(2s t+\frac{s^2}{2}t^2\right)dt.\eeas
The pre-Schwarzian norm of $f_3$ is given by
\beas \Vert P_{f_3}\Vert =\sup_{z\in\mathbb{D}}(1-|z|^2)|P_{f_3}(z)|=\sup_{z\in\mathbb{D}}\left|\frac{(1-|z|^2)((1+sz)^2-1)}{z}\right|.\eeas
On the positive real axis, we have
\beas \sup_{0\leq r<1}\left(\frac{(1-r^2)((1+sr)^2-1)}{r}\right)=\frac{2 \left(\sqrt{3 s^2+4}+4\right) \left(3 s^2+2 \sqrt{3 s^2+4}-4\right)}{27 s}.\eeas
Thus, we have
\beas\Vert P_{f_3}\Vert =\frac{2 \left(\sqrt{3 s^2+4}+4\right) \left(3 s^2+2 \sqrt{3 s^2+4}-4\right)}{27 s}.\eeas
This completes the proof.
\end{proof}
\section*{Declarations}
\noindent{\bf Acknowledgment:} The work of the second author is supported by University Grants Commission (IN) fellowship (No. F. 44 - 1/2018 (SA - III)).\\[2mm]
{\bf Conflict of Interest:} The authors declare that there are no conflicts of interest regarding the publication of this paper.\\[2mm]
{\bf Availability of data and materials:} Not applicable\\[2mm]
{\bf Authors' contributions:} All authors contributed equally to the investigation of the problem, and all authors have read and approved the final manuscript.

\end{document}